\documentclass[12pt]{article}

\usepackage{graphicx}
\usepackage{framed}
\usepackage[normalem]{ulem}
\usepackage{amsmath}
\usepackage{amsthm}
\usepackage{amssymb}
\usepackage{amsfonts}
\usepackage{enumerate}
\usepackage[utf8]{inputenc}
\usepackage[top=1 in,bottom=1in, left=1 in, right=1 in]{geometry}
\usepackage{xcolor}
\usepackage{hyperref}
\usepackage{tikz}
\usepackage[procnames]{listings}
\usepackage[toc]{appendix}

\usetikzlibrary{arrows}

\tikzset{
	vertex/.style={circle,draw,minimum size=1.5em},
	edge/.style={->,> = latex'}
}

\newcommand{\R}{\mathbb{R}}
\newcommand{\N}{\mathbb{N}}

\newcommand{\Z}{\mathbb{Z}}
\newcommand{\Fp}{\mathbb{F}_p}

\newtheorem{theorem}{Theorem}[section]
\newtheorem{corollary}[theorem]{Corollary}
\newtheorem{lemma}[theorem]{Lemma}
\newtheorem{prop}[theorem]{Proposition}

\theoremstyle{definition}
\newtheorem{definition}[theorem]{Definition}
\newtheorem{example}[theorem]{Example}

\newtheorem{remark}[theorem]{Remark}

\title{Arithmetic of idempotents in $\Z / m\Z$}
\author{Kelly Isham and Laura Monroe \thanks{This publication is unclassified and has been assigned LANL identifier LA-UR-20-23432. A portion of this work was performed at the Ultrascale Systems Research Center (USRC) at Los Alamos National Laboratory, supported by the U.S. Department of Energy contract DE-FC02-06ER25750. The first author was supported in part by an appointment with the National Science Foundation (NSF) Mathematical Sciences Graduate Internship (MSGI) Program sponsored by the NSF Division of Mathematical Sciences. }}
\date{}

\begin{document}

	\maketitle
	\begin{abstract}
		Idempotent elements are a well-studied part of ring theory, with several identities of the idempotents in $\Z / m\Z$ already known. Although the idempotents are not closed under addition, there are still interesting additive identities that can be derived and used.
		
		In this paper, we give several new identities on idempotents in $\Z / m\Z$. 
		We relate finite sublattices over $\Z / k\Z$ for all integers $k$ to an infinite lattice that is embedded in the divisibility lattice on $\N$ and to each other as sublattices of this infinite lattice. Using this relation, we generalize several identities on idempotents in $\Z/ m\Z$ to those involving idempotents related to these finite sublattices.
		
		Finally, as an application of the above idempotent identities, we derive an algorithm for calculating modular exponentiation over $\Z/ m\Z$.
		
	\end{abstract}

\section{Introduction}
An element $d$ in a ring $R$ is idempotent if $d^2 = d$. In 1882, Peirce published \cite{peirce}, which demonstrated the important role that idempotent and nilpotent elements play in ring theory. Since then, researchers have tried to understand properties of idempotents in both commutative and noncommutative rings. Many important results involve sets of idempotents that are \emph{primitive} and \emph{orthogonal}. An idempotent $d$ is primitive if $dR$ is indecomposable. A pair $\{d,e\}$ of idempotents are orthogonal if $de = 0$ in $R$. In this paper, we focus on idempotents in $R = \Z/ m\Z$. We give more attention to the idempotents that are not orthogonal or primitive.

While additive identities involving the primitive orthogonal idempotents modulo $m$ are known (see for example the summary paper on idempotents \cite{porubsky}), we have not found many identities involving arbitrary idempotents in the literature. In this paper, we provide several new identities. Further, we establish a connection between the infinite divisibility lattice and idempotents by defining \emph{consistent} sublattices. We generalize the idempotent identities modulo $m$ to identities involving the elements of the consistent sublattices modulo $g$ where $g$ is the supremum of the given sublattice.

Idempotents play an important role in understanding exponentiation modulo $m$. Several authors including Hewitt and Zuckerman in \cite{hewitt} and Schwarz in \cite{schwarz} have studied the orbits $a, a^2, \ldots$ in $\Z / m\Z$ and have shown that eventually $a^k = d$ for some idempotent element $d \in \Z / m\Z$. In \cite{ish_comps}, we define the sequential power graph with vertex set $V = \Z/ m\Z$ and edges $(a,b)$ such that $a = c^i$ and $b =c^{i+1}$ for some $c \in \Z / m\Z$ and $i \in \N$. Also in \cite{ish_comps}, we show that much of the structure of this graph follows directly from the idempotents. In Section \ref{app_sect}, we use some additive identities on idempotents and results about the components of the sequential power graph to provide an algorithm based on the Chinese Remainder Theorem for computing $b^e \pmod{m}$, assuming the factorization of $m$ is known. We hope that the identities in this paper will be useful in other applications as well.

\section{Notation} \label{notation}

Throughout this paper, we assume that the factorization of $m$ is known. Theorem \ref{idem_struc} provides the structure of idempotents modulo $m$. Let $I = \{i_1, \ldots, i_s\}$. To introduce notation early on, this theorem shows that $d_I = a_I p_{i_1}^{e_{i_1}} \cdots p_{i_s}^{e_{i_s}}$ such that $d_I \equiv 1\pmod{p_j^{e_j}}$ for all $j \in R \setminus I$. The idempotent $d_I$ is determined by the set $I$.\\

We define certain notations that we will use throughout this paper.
\begin{itemize}
	\item $m = p_1^{e_1} \cdots p_r^{e_r}$ is the modulus. 
	\item $R = \{1, 2, \ldots, r\}$ is the set of all indices referring to the prime powers dividing $m$. 
	\item $I = \{i_1, \ldots, i_s\} \subset R$ defines the primes that divide the idempotent $d_I$.
	\begin{itemize}
		\item[$\centerdot$] $g_I=$ gcd$(d_I,m)= p_{i_1}^{e_{i_1}} \cdots p_{i_s}^{e_{i_s}}$ is the greatest common divisor of $d_I$ and $m$.
		\item[$\centerdot$]  $d_I = a_Ig_I$ is the idempotent corresponding to $I$, with $a_I$ relatively prime to $\frac{m}{g_I}=g_{R \setminus I}$.
	\end{itemize}
	\item $R \setminus I = \{i_{s+1}, \ldots, i_r\} \subset R$ defines the set dual to $I$ in the following sense
	\begin{itemize}
		\item[$\centerdot$] $g_{R \setminus I}=$ gcd$(d_{R \setminus I},m)= p_{i_{s+1}}^{e_{i_{s+1}}} \cdots p_{i_{r}}^{e_{i_{r}}}$ is the greatest common divisor of $d_{R \setminus I}$ and $m$.
		\item[$\centerdot$] $d_{R \setminus I} = a_{R \setminus I}g_{R \setminus I}$ is the idempotent corresponding to $R \setminus I$, with $a_{R \setminus I}$ relatively prime to $\frac{m}{g_{R \setminus I}}=g_I$.
	\end{itemize}
	
\end{itemize}

\section{The idempotents of $\Z / m\Z$}
\label{idem_section_1}
\subsection{Known identities}

The idempotents of $\Z / m\Z$ have been well-studied. We begin with the structure theorem for idempotents modulo $m$. 

\begin{theorem}	\label{idem_struc}\cite[Theorem 2.2]{hewitt}  Let $I = \{i_1, \ldots, i_s\}\subseteq R$. Then the idempotents of $\Z/ m\Z$ are of the form $a_Ig_I \pmod{m}$ for any $a_I$ satisfying $a_I  g_I\equiv 1 \pmod{\frac{m}{g_I}}$ such that $\text{gcd}(a_I, \frac{m}{g_I}) =1.$
\end{theorem}
\begin{corollary}\cite[Theorem 2.2]{hewitt}
	There are $2^r$ idempotent elements in $\Z / m\Z$.
	\label{num_idem}
\end{corollary}

It is well-known (see e.g. \cite{porubsky}) that $d_{R \setminus \{i\}}$ for $i = 1, \ldots, r$ are the primitive orthogonal idempotents of $\Z / m\Z$. The following proposition follows immediately from this fact.
\begin{prop} \cite[Proposition 41]{peirce}, \cite[Page 1237]{porubsky}
	The direct sum $d_{R \setminus \{1\}}(\Z / m\Z)^\times \times \cdots \times d_{R \setminus \{r\}}(\Z / m\Z)^\times $ is the Peirce decomposition of the ring $\Z / m\Z$, up to order of idempotents. 
\end{prop}




Further, the idempotents under multiplication are well-understood. 

\begin{prop}\cite[Corollary 2.21]{hewitt}
	The idempotents of $\Z / m\Z$ form a monoid under multiplication.
\end{prop}

\begin{prop} \cite[Corollary 2.21]{hewitt}\label{set_idem}
	Let $\{I\}$ be a set of subsets of $R$. Then $\prod_{I} d_I \equiv d_{\bigcup I} \pmod{m}.$
\end{prop}
\begin{remark}
	All idempotents besides 1 are zero divisors, so cannot be units. Proposition \ref{set_idem} shows in a different way that the all idempotents not equal to 1 cannot be units.
\end{remark}

\subsection{Lattice of idempotents}
\label{idem_lattice_sect}
There is a standard partial ordering on idempotents $d, e$ in a ring $R$ given by $d \leq e$ if and only if $de = ed = e$ in $R$. If $R$ is commutative, the idempotents form a lattice. While there is a standard definition of the meet and join of this lattice, we will instead give an equivalent definition when $R = \Z / m\Z$.  By Proposition \ref{set_idem}, if $d_I$ and $d_J$ are idempotents in $\Z/ m\Z$, then $d_I d_J \equiv d_J \pmod{m}$ if and only if $I \cup J = J$. Therefore, we can express this partial ordering in a simpler way. This partial order induces a lattice structure on the idempotents modulo $m$.
\begin{definition}
	Let $d_I, d_J$ be idempotents in $\Z/ m\Z$. Then $d_I \leq d_J$ if and only if $I \subseteq J$. The idempotents in $\Z/m\Z$ form a lattice with $d_I \vee d_J = d_{I \cup J}$ and $d_I \wedge d_J = d_{I \cap J}$. The infimum is $d_{\{\}} =1$ and the supremum is $d_R = 0$. 
\end{definition}

\begin{example}
	Let $m = p_1^{e_1} p_2^{e_2} p_3^{e_3}$. Then the lattice for idempotents in $\Z/ m\Z$ is given by
	
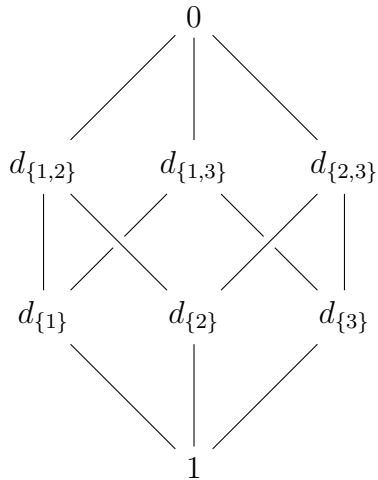
\begin{center}
		\begin{tikzpicture}
		\node (max) at (0,4) {$0$};
		\node (a) at (-2,2) {$d_{\{1,2\}}$};
		\node (b) at (0,2) {$d_{\{1,3\}}$};
		\node (c) at (2,2) {$d_{\{2,3\}}$};
		\node (d) at (-2,0) {$d_{\{1\}}$};
		\node (e) at (0,0) {$d_{\{2\}}$};
		\node (f) at (2,0) {$d_{\{3\}}$};
		\node (min) at (0,-2) {$1$};
		\draw (min) -- (d) -- (a) -- (max) -- (b) -- (f)
		(e) -- (min) -- (f) -- (c) -- (max)
		(d) -- (b);
		\draw[preaction={draw=white, -,line width=6pt}] (a) -- (e) -- (c);
		\end{tikzpicture}
	\end{center}
	
	Observe that this lattice is isomorphic to the lattice of subsets of $\{1,2,3\}$ under the partial ordering $\subseteq$.
\end{example}

\begin{definition}
	The $k^{th}$ level of the idempotent lattice is the row in the lattice consisting of elements $d_I$ where $|I| = k$. We say the $0^{th}$ level of the lattice is $\{1\}$. An idempotent is a \emph{top-level idempotent} if it is an element of the $(r-1)^{st}$ row.
\end{definition}

\begin{remark} \label{two_defs_lattice}
	There is another way of defining the lattice of idempotents. Recall that $d_I \equiv g_I a_I \pmod{p}$ where $g_I$ is the product of the prime powers from index set $I$ dividing $m$. Define a partial ordering $d_I \leq d_J$ if and only if $g_I | g_J$. The meet is $d_I \vee d_J = \text{lcm}(g_I, g_J)$ and the join is $d_I \wedge d_J = \text{gcd}(g_I, g_J)$. This lattice is isomorphic to the subset lattice and thus to our above definition of the idempotent lattice modulo $m$. To see this, notice that $\text{lcm}(g_I, g_J) = g_{I \cup J}$ and $\text{gcd}(g_I, g_J) = g_{I \cap J}$. 
\end{remark}

We introduce the lattice of idempotents here to provide motivation for our additive identities below. In particular, the example of an idempotent lattice will be useful as we provide additive identities across a level of the lattice (Proposition \ref{idem_sum_k}) or across an entire sublattice (Proposition \ref{sublattice_sum})

\subsection{Addition on idempotents}
\label{add_identities}
Although the idempotents form a multiplicative monoid, and not a group or a ring, there is still a good amount of structure to this set. Addition on idempotents is not as simple as multiplication. For one thing, the idempotents are not closed under addition; still, there are some additive identities that can be proved. 

Much is known, classically and otherwise, about the primitive orthogonal idempotent elements $(d_{R \setminus \{i\}}$) including several of the additive identities that we cite below. However, to our knowledge, there are very few identities involving the other levels of idempotents in the literature. We provide several new additive identities involving these other level idempotents.

\begin{prop}\cite{peirce} \label{dual_idems}
	Let $d_I$ be the idempotent in $\Z/ m\Z$ defined by set $I$. Then $1-d_I \pmod m$ is the idempotent $d_{R \setminus I} $.
\end{prop}
\begin{corollary}\cite{peirce} \label{complement_idems}
	Let $d_I$ and $d_{R \setminus I}$ be idempotents in $\Z/ m\Z$ defined by sets $I$ and $R \setminus I$ respectively.  Then $d_I+d_{R \setminus I} \equiv 1 \pmod m$.
\end{corollary}

Now that we have established a few basic addition formulas, we note that we can impose subtraction on idempotents as well. Observe that division of idempotents is undefined since all idempotents not equal to 1 are zero divisors.
\begin{corollary}
	$d_J- d_I \equiv d_J + d_{R \setminus I} - 1 \pmod{m}$. 
\end{corollary}
\begin{prop} \cite[Lemma 11]{porubsky} \label{sum_idems}
	Let $R \ne \emptyset$. Let $J = \{j_1, \ldots, j_s\} \subset R$ and let $d_J$ be idempotent.
	Then $\sum_{i \in R \setminus J} d_{R \setminus \{i\}} \equiv d_J \pmod m$.
\end{prop}



\begin{prop}\cite[Lemma 2.3]{schwarz} \label{idem_sum_1}
	\label{sum_is_1}
	Let $R \ne \emptyset$. Then $\sum_{i=1}^r d_{R \setminus \{i\}} \equiv 1 \pmod m$.
\end{prop}


We now introduce an important additive identity. This will be useful in the proofs of the theorems below.
\begin{prop}
	\label{union_sum}
	$d_I + d_J \equiv d_{I \cup J } + d_{I \cap J}\pmod{m}.$
\end{prop}
\begin{proof}
	By Proposition \ref{sum_idems}, we have
	\begin{align*}
	d_I & \equiv \sum_{i \in R \setminus I} d_{R \setminus \{i\}}\pmod{m}\\
	&\equiv \sum_{i \in R \setminus (I \cap J)} d_{R \setminus \{i\}} - \sum_{i \in I \setminus J }d_{R \setminus \{i\}} \pmod{m}
	\end{align*}
	Focusing on the last term, by properties of sets we obtain
	\begin{align*}
	\sum_{i \in I \setminus J }d_{R \setminus \{i\}}  & \equiv \sum_{i \in R} d_{R \setminus \{i\}}  - \sum_{i \in R \setminus (I \cup J)} d_{R \setminus \{i\}} - \sum_{i \in J}d_{R \setminus \{i\}} \pmod{m}\\
	&\equiv 1  - d_{I \cup J} - d_{R \setminus J}\pmod{m}
	\end{align*}
	
	Putting these together, 
	\begin{align*}
	d_I &\equiv d_{I \cap J} - \left( 1  - d_{I \cup J} - d_{R \setminus J}\right)\pmod{m}\\
	&\equiv d_{I \cap J} - 1 + d_{I \cup J} + d_{R \setminus J}\pmod{m}
	\end{align*}
	
	Therefore we obtain
	\begin{align*}
	d_I + d_J &\equiv d_{I \cap J} - 1 + d_{I \cup J} + d_{R \setminus J} + d_J\\
	& \equiv d_{ I \cap J} + d_{I \cup J} \pmod{m}
	\end{align*}
	by Proposition \ref{dual_idems}.
\end{proof}

\begin{corollary} 
	Let $d_I$ and $d_J$ be idempotent. If $m$ is odd, $d_I+ d_J$ is idempotent if and only if $d_I + d_J \equiv d_{I \cap J} \pmod{m}$.
\end{corollary}
\begin{proof}
	Observe that $(d_I + d_J)^2 \equiv d_I + d_J$ if and only if $2d_Id_J \equiv 0 \pmod{m}$. Since $m$ is odd, $d_I + d_J$ is idempotent if and only if $d_I d_J \equiv 0 \pmod{m}$. This occurs if and only if $d_{I \cup J} \equiv 0 \pmod{m}$. 
\end{proof}




\begin{prop}
	If $I_1, \ldots, I_k$ are pairwise disjoint and $I_1 \cup \cdots \cup I_k = J$, then $\sum_{i=1}^k d_{I_i} \equiv k-1 + d_J\pmod{m}$.
\end{prop}
\begin{proof}
	We will prove this by induction on $k$.\\
	
	If $k =1$, then $d_{I_1} \equiv 1-1 + d_{I_1} \pmod{m}$.\\
	
	Suppose $\sum_{i=1}^{k-1} d_{I_i} \equiv (k-2) + d_{J'} \pmod{m}$ where $J' = I_1 \cup \cdots \cup I_{k-1} $. Then 
	$$
	\sum_{i=1}^{k} d_{I_i} \equiv (k-2) + d_{J'} + d_{I_k} \equiv k-2 + 1 + d_J \pmod{m}
	$$
	by Proposition \ref{union_sum} and since $J' \cap I_k = \emptyset$ and $J' \cup I_k = J$.
	
\end{proof}
\begin{corollary}
	If $I_1, \ldots, I_k$ are pairwise disjoint and $I_1 \cup \cdots \cup I_k = R$, then $\sum_{i=1}^k d_{I_i} \equiv k-1 \pmod{m}$.
\end{corollary}
\begin{corollary}
	Let $I \subseteq R$ with $|I| = k>1$. We have $\sum_{i \in I} d_{\{i\}} \equiv k-1  + d_{I} \pmod{m}.$
\end{corollary}




The following theorem provides additive identities for the $k^{th}$ level of idempotents (all $d_J$ such that $|J| = k$) for all $k = 0, \ldots, r-1$. 
\begin{theorem} \label{idem_sum_k}
	Let $|R| = r > 0$, and let $0 \le k < r$. Then $\sum_{|J|=k} d_J \equiv \binom{r-1}{k} \pmod{m}$.
\end{theorem}
\begin{proof}
	For $k=r-1$, this proposition is equivalent to Proposition \ref{idem_sum_1}. 
	\newline
	Let $k<r-1$. If $|J|=k$, $d_J \equiv \sum d_I \pmod m$, where $I = R \setminus\{i\}$ for each $i \in  R \setminus J$, by Proposition \ref{sum_idems}. These $I$ are the $(r-1)$-element subsets of $R$ such that $J \subset I$, so each $d_I$ is a summand of exactly $\binom{r-1}{k}$ of the $d_J$. This means that the summation $\sum d_I = 1$ occurs $\binom{r-1}{k}$ times in $\sum_{|J|=k} d_J$, so $\sum_{|J|=k} d_J = \binom{r-1}{k}$.
\end{proof}

\begin{lemma}
	\label{one_level}
	Suppose $|I| = k$. We sum over all $d_J$ so that $|J| = k-n$ and $J \subset I$ where $0 < n< k$. In other words, we sum over all elements in the idempotent lattice $n $ levels below $d_I$ that eventually connect to $d_I$. This gives $$\sum_{\substack{|J| = k-n\\ J \subset I}} d_J  \equiv \binom{k-1}{n} + \binom{k-1}{n-1} d_I\pmod{m}.$$
\end{lemma}
\begin{proof}
	\begin{align*}
	\sum_{\substack{|J| = k-n\\ J \subset I}} d_J &= \sum_{\substack{|J| = k-n\\ J \subset I}}  \sum_{i \in R \setminus J} d_{R \setminus \{i\}}
	\end{align*}
	by Proposition \ref{sum_idems}. Note that 
	$$
	\sum_{i \in R \setminus J} d_{R \setminus \{i\}} = \sum_{i \in R \setminus I} d_{R\setminus \{i\}} +\sum_{i \in (I \setminus J)} d_{R \setminus \{i\}}
	$$
	and that $|I \setminus J| = n$. Further, there are $\binom{k}{n} $ subsets $J$ satisfying the conditions. Therefore
	\begin{align*}
	\sum_{\substack{|J| = k-n\\ J \subset I}} d_J  &\equiv \binom{k}{n} \sum_{i \in R \setminus I} d_{R\setminus \{i\}}  + \sum_{\substack{|J| = k-n\\ J \subset I}} \sum_{i \in (I \setminus J)} d_{R \setminus \{i\}}\\
	&\equiv \binom{k}{n} d_I + 	\sum_{\substack{|J| = k-n\\ J \subset I}} \sum_{i \in (I \setminus J)} d_{R \setminus \{i\}}
	\end{align*}
	We now focus on understanding this last summation. When taking the summation of $|J| = k-n, J \subset I$, every term in the summation will contain all $d_{R \setminus \{i\}}$ such that $i \not \in I$. That is, every term will contain $d_{R \setminus I}$. Each element in $R \setminus \{i\}$ for $i \in I$ occurs $\frac{\binom{k}{n} (k-n)}{k}$ times since there are $\binom{k}{n}$ terms, each containing $k-n$ terms from $I$ and all occur equally often. Thus,
	\begin{align*}
	\sum_{\substack{|J| = k-n\\ J \subset I}} \sum_{i \in (I \setminus J)} d_{R \setminus \{i\}}&= \frac{\binom{k}{n} (k-n)}{k} d_{R \setminus I} 
	\end{align*}
	
	Putting these together, 
	\begin{align*}
	\sum_{\substack{|J| = k-n\\ J \subset I}} d_J &\equiv\binom{k}{n} d_I +\frac{\binom{k}{n} (k-n)}{k} d_{R \setminus I}   \\
	&\equiv \left(\binom{k}{n} -\frac{\binom{k}{n} (k-n)}{k} \right) d_I +\frac{\binom{k}{n} (k-n)}{k}   \\
	&\equiv \binom{k-1}{n-1}d_I + \binom{k-1}{n}\pmod{m}.
	\end{align*}
\end{proof}

The last lemma seems abstract, so consider the following corollary when $n = 1$. 
\begin{corollary}
	Suppose $|I| = k$. We sum over all $d_J$ so that $|J| = k-1$ and $J \subset I$. In other words, we sum over all elements in the idempotent lattice in the level directly below $d_I$ that are connected to $d_I$. This gives the identity $$\sum_{|J| = k-1, J \subset I} d_J  \equiv (k-1)+d_I \pmod{m}.$$
\end{corollary}

With Lemma \ref{one_level} established, we can now determine the summation of all idempotents below $d_I$ in terms of $I$ only. We state the theorem modulo $m$ and provide a corollary for determining the summation modulo $g_I$ as well. 
\begin{theorem} \label{sublattice_sum}
	Let $|I| = k$. Then $\sum_{J \subseteq I} d_J \equiv 2^{k-1}(1+d_I)\pmod{m}$. 
\end{theorem}
\begin{proof}
	By Lemma \ref{one_level}, $$\sum_{\substack{|J| = k-n\\ J \subset I}} d_J \equiv \binom{k-1}{n} + \binom{k-1}{n-1} d_I\pmod{m}.$$ When $J = \emptyset$, the idempotent is $d_J = 1$ and when $J = I$, the idempotent is $d_J = d_I$. Therefore
	\begin{align*}
	\sum_{J \subseteq I} d_J &\equiv 1 + d_I + \sum_{n=1}^{k-1} \sum_{\substack{|J| = k-n\\ J \subset I}} d_J \pmod{m}\\
	&\equiv 1+ d_I+ \sum_{n=1}^{k-1} \binom{k-1}{n} +d_I \sum_{n=1}^{k-1}\binom{k-1}{n-1}  \pmod{m}\\
	& \equiv 2^{k-1} + d_I \sum_{n=0}^{k-1} \binom{k-1}{n}\pmod{m}\\
	& \equiv 2^{k-1} + 2^{k-1} d_I \pmod{m}
	\end{align*}
\end{proof}

\begin{corollary}
	Let $|I| = k$. Then $\sum_{J \subseteq I} d_J \equiv 2^{k-1}\pmod{g_I}$. 
\end{corollary}

\begin{corollary}
	$\sum_{I\subseteq R} d_I \equiv 2^{r-1} \pmod{m}$.
\end{corollary}

\section{Sublattices of the infinite divisibility lattice}
In Section \ref{idem_lattice_sect}, we defined the lattice of idempotents in $\Z / m\Z$. In this section, we seek to generalize the notion of an idempotent lattice. By Remark \ref{two_defs_lattice}, the idempotent lattice modulo $m$ can be defined in terms of the elements $g_I = \text{gcd}(d_I, m)$ for each $I \subseteq R$. We generalize this definition of the idempotent lattice modulo $m$ to both finite and infinite lattices. These will be subsets of the infinite divisibility lattice on $\N$, which we call \emph{consistent sublattices}. Since the only idempotents in $\Z$ are 0 and 1, the notion of an infinite idempotent lattice in $\N$ does not make sense. However, the construction of the infinite divisibility lattice on $\N$ permits us to relate an infinite set of finite lattices over $\Z / m\Z$ for integers $m$ to each other as sublattices of this infinite lattice.

\subsection{An infinite lattice related to the idempotents modulo $m$}

\begin{definition}
	The infinite divisibility lattice on $\N$ is given by the partial ordering $a \leq b$ if and only if $a | b$. The meet of two elements $a,b$ is $a \vee b = \text{lcm}(a,b)$ and the join is $a \wedge b = \text{gcd}(a,b)$. 
\end{definition}

\begin{definition}
	List the primes in increasing order with $p_1 = 2$. Let $S  \subseteq \N$ and $T \subseteq S$. Let $E$ be an infinite sequence $E = (e_1, e_2, \ldots)$ of elements in $\N \cup \{0\}$ so that $e_i > 0$ for all $i \in S$ and $e_i = 0$ for all $i \not \in S$. A sublattice of the infinite divisibility lattice is \textit{consistent} if the elements of this lattice are exactly the integers $g_K = \prod_{k \in K} p_k^{e_k}$ for all $T \subseteq K \subseteq S$. We denote this sublattice $\mathcal{L}_{E, S, T}$.
	
\end{definition}

\begin{remark}
	Suppose $S= \{s_1, \ldots, s_r\}$ is finite and $T = \emptyset$. Notice that $E$ contains only finitely many nonzero elements. Set $m = p_{s_1}^{e_{s_1}} \cdots p_{s_r}^{e_{s_r}}$ and let $d_K$ be the idempotent modulo $m$ corresponding to set $K \subseteq S$. The consistent sublattice defined by sets $E$, $S$ and $T= \emptyset$ contains exactly the elements $g_K = \text{gcd}(m, d_K)$. Therefore $\mathcal{L}_{E, S, T}$ is isomorphic to the a sublattice of the idempotent lattice modulo $m$. Since $g_K$ depends only on the sets $E, S$, and $T$, it makes sense to extend the notion of the elements $g_K$ when $S$ is infinite.
	
	Suppose $S$ is infinite, $T \subseteq S$, and $E = (e_1, e_2, \ldots)$ with $e_i > 0$ if $i \in S$ and $e_i =0$ if $i \not \in S$. Write the supremum $g_S = \prod_{s\in S} p_s^{e_s}$ and the infimum $g_T = \prod_{t \in T} p_t^{e_t}$. Let $\mathcal{L}_{E,S, T}$ be the consistent sublattice with supremum $g_S$ and infimum $g_T$. Then $b \in \mathcal{L}$ if and only if
	$
	b =g_T \prod_{j \in K} p_j^{e_j}
	$ 
	for some $K \subseteq S \setminus T.$ 
\end{remark}

\begin{definition}
	A consistent sublattice of the infinite divisibility lattice is \textit{maximally consistent} if $S= \N$ and $T =\emptyset$, that is, if the supremum is divisible by every prime and the infimum is 1.
\end{definition}

We can uniquely determine a maximally consistent sublattice from the first level, which consists of prime powers. It is clear that there is a bijection between sequences of positive integers and maximally consistent sublattices; each sequence of positive integers $(e_1, e_2, \ldots)$ determines a distinct maximally consistent sublattice by setting the first level equal to $\{p_1^{e_1} , p_2^{e_2}, \ldots\}$ and this correspondence is surjective.
\begin{prop}
	The set of maximally consistent sublattices of $\mathcal{L}$ is uncountable and has cardinality of $\R$.
\end{prop}
\begin{proof}
	The set of maximally consistent sublattices of $\mathcal{L}$ has the cardinality of $\R$. From the remark above, we know that each maximally consistent sublattice is uniquely determined by the sequence $(e_1, e_2, \ldots)$ of positive integers. There is a bijection of the set of such sequences to $\R$ by sending $(e_1, e_2, \ldots)$ to $.(e_1-1)(e_2-1)\cdots$. The inverse of this function is $.b_1b_2\ldots\mapsto (b_1+1, b_2 + 1, \ldots)$. 
\end{proof}

\begin{remark}
	Let the idempotent lattice modulo $m$ be denoted $\mathcal{L}_m$. Observe that there is a bijection between finite maximally consistent lattices and lattices of the form $\mathcal{L}_m$.
\end{remark}
\subsection{Sublattices of the finite consistent lattices}

Let $\mathcal{L}_{E, S, T}$ be a finite consistent lattice and let $S' \supseteq S$ be a finite set. Then $\mathcal{L}_{E, S,T}$ can be extended to the finite maximally consistent lattice $\mathcal{L}_{E, S', \emptyset}$. As noted in the remark above, $L_{E,S', \emptyset}$ is isomorphic to the lattice $\mathcal{L}_m$ of idempotents in $\Z/ m\Z$ where $m = g_{S'}$ is the supremum of $\mathcal{L}_{E, S',T}$. Therefore finite consistent sublattices are isomorphic to sublattices of $\mathcal{L}_m$. We denote this sublattice of $\mathcal{L}_m$ by $\mathcal{L}_{m, S,T}$ since the sequence $E$ is clear given $m$.

In the next section, we will consider idempotent identities on sublattices of $\mathcal{L}_m$. We use notions from consistent lattices throughout.

\section{Idempotents in sublattices}
Many of the theorems we have proved in Section \ref{idem_section_1} have analogues in the consistent lattices $\mathcal{L}_{m,S,T}$ of $\mathcal{L}_m$. For notation, we set $g_{S}$ to be the supremum of $\mathcal{L}_{m , S, T}$ and we let $d_S$ be the corresponding idempotent modulo $m$. Recall that $m = g_{S'}$ for some finite set $S' \supseteq S$. Similarly, $g_T$ is the infimum and $d_T$ is the corresponding idempotent modulo $m$.

\subsection{Multiplication on idempotents in finite consistent lattices}
\begin{prop} \label{set_idem_gen}
	Let $\mathcal{L}_{m,S,T}$ be a sublattice of $\mathcal{L}_m$. Let $\{I\}$ be a set of subsets of $S$ that contain $T$. Then $\prod_{I} d_I\equiv  d_{\bigcup I} \pmod{g_S}.$
\end{prop}
\begin{proof}
	Since $g_S | m$ and this equivalence holds $\pmod{m}$, then this equivalence must also hold $\pmod{g_S}$. 
\end{proof}

\subsection{Addition on idempotents in finite consistent lattices}
In this section, we generalize the additive identities established in Section \ref{add_identities}. Observe that for any element $g_I \in \mathcal{L}_{m, S, T}$, $T \subseteq I \subseteq S$. Therefore all sets considered in the following propositions and theorems must contain $T$. 

\label{gen_id_add}
\begin{prop}
	\label{union_int_gen}
	Let $\mathcal{L}_{m,S,T}$ be a sublattice of $\mathcal{L}_m$ and let $T \subseteq I, J \subseteq S$. Then $d_I + d_J \equiv d_{I \cup J } + d_{I \cap J}\pmod{g_S}.$
\end{prop}
\begin{proof}
	This holds modulo $m$ and $g_S | m$. 
\end{proof}

\begin{prop}
	Let $\mathcal{L}_{m,S,T} \subseteq \mathcal{L}_m$. Let $I_1, \ldots, I_k$ be sets such that $T \subseteq I_\ell \subseteq S$ for all $1 \leq \ell \leq k$ and $I_\ell \cap I_j = T$ for all $ \ell \ne j$. Suppose that $I_1 \cup \cdots \cup I_k = J$. We have $\sum_{i=1}^k d_{I_i} \equiv (k-1)d_T + d_J\pmod{g_S}$.
\end{prop}
\begin{proof}
	We will prove this by induction on $k$.\\
	
	If $k =1$, then $d_{I_1} \equiv (1-1)d_T + d_{I_1} \pmod{g_S}$.\\
	
	Let $J' = I_1 \cup \cdots \cup I_{k-1}$ and assume that $\sum_{i=1}^{k-1} d_{I_i} \equiv (k-2)d_T + d_{J'} \pmod{g_S}$. Observe that $J' \cap I_k =T$ and $J' \cup I_k = J$. Applying Proposition \ref{union_int_gen} gives
	$$
	\sum_{i=1}^{k} d_{I_i} \equiv (k-2) d_T+ d_{J'} + d_{I_k} \equiv (k-2) d_T + d_T+ d_J \pmod{g_S}.
	$$

\end{proof}
\begin{corollary}
	Let $\mathcal{L}_{m,S,T}$ be a sublattice of $\mathcal{L}_m$. If $I_1, \ldots, I_k$ are sets such that $T \subseteq I_\ell \subseteq S$ for all $1 \leq \ell \leq k$, $I_\ell \cap I_j = T$ for all $\ell \ne j$, and $I_1 \cup \cdots \cup I_k = S$, then $\sum_{i=1}^k d_{I_i} \equiv (k-1)d_T \pmod{g_S}$.
\end{corollary}
\begin{corollary}
	Let $\mathcal{L}_{m,S,T}$ be as above. Let $T \subseteq I \subseteq S$ with $|I| = k>1$. Then $\sum_{i \in I} d_{T \cup \{i\}} \equiv (k-1)d_T  + d_{I} \pmod{g_S}.$
\end{corollary}

The following corollary is a generalization of Corollary \ref{complement_idems}. Here, the infimum is $d_T$, so if $g_I = \text{gcd}(d_I, g_S)$ and $g_J = \text{gcd}(d_J, g_S)$ are elements in $\mathcal{L}_{m,S,T}$, we require $T \subseteq J$ as well. Therefore, $J = S \setminus I$ does not make sense in this setting. The correct generalization is $J = \left(S \setminus I \right)\cup T = S \setminus (I \setminus T).$

\begin{corollary}\label{dual_sum_gen}
	Let $T \subseteq I \subseteq S$ and $\mathcal{L}_{m,S,T}\subseteq \mathcal{L}_m$.  We have $d_I+d_{S \setminus (I \setminus T)} \equiv d_{ T} \pmod{g_S}$.
\end{corollary}
\begin{lemma}\label{sub_set_sum_gen}
	Suppose $S \ne \emptyset$. Let $\mathcal{L}_{m,S,T} \subseteq \mathcal{L}_m$, and $T \subset I \subseteq S$. Then $\sum_{i \in I \setminus T} d_{S \setminus \{i\}} \equiv d_{S \setminus \left(I \setminus T\right) }\pmod{g_S}$.
\end{lemma}
\begin{proof}
	We will prove the lemma using induction on $|I|> |T|$. If $I= T \cup \{i_1\}$ for some $i_1 \in S \setminus T$, then $$\sum_{i \in I \setminus T} d_{S \setminus \{i\}} = d_{S \setminus \{i_1\}} = d_{S \setminus \left(I\setminus T\right)}.$$
	
	
	Now suppose that this equivalence holds for $|I' \setminus T| =k-1$. Then $\sum_{i \in I' \setminus T} d_{S \setminus \{i\}} \equiv d_{S \setminus \left(I'\setminus T\right)} \pmod{g_S}$. Pick an element $i_k \in S \setminus T$ so that $i_k \not\in I'$ and set $I  = I' \cup \{i_k\}$. Then
	\begin{align*}
	\sum_{i \in I \setminus T} d_{S \setminus \{i\}} &= \sum_{i\in I'\setminus T} d_{S \setminus \{i\}} + d_{S \setminus \{i_k\}} \\
	&\equiv d_{S \setminus \left(I' \setminus T\right)} + d_{S \setminus \{i_k\}} \pmod{g_S}\\
	&\equiv d_S + d_{S \setminus \left(I\setminus T\right)} \pmod{g_S}\\
	&\equiv d_{S \setminus \left(I\setminus T\right)} \pmod{g_S}.
	\end{align*}
	The penultimate line is due to Proposition \ref{union_int_gen}. The last line holds because $g_S | d_S$. 
\end{proof}

\begin{prop}\label{sum_idems_gen}
	Suppose $S \ne \emptyset$ and let $\mathcal{L}_{m,S,T}$ be as above. If $T \subseteq J \subset S$,  $\sum_{i \in S \setminus J} d_{S \setminus \{i\}} \equiv d_J \pmod{g_S}$.
\end{prop}
\begin{proof}
	Use Lemma \ref{sub_set_sum_gen} and set $I = S \setminus J$.
\end{proof}


\begin{prop} \label{idem_sum_k_gen}
	Let $\mathcal{L}_{m,S,T}$ be as above with $|S| = s \geq 1$ and $|T| = t$. Let $t \le k < s$. Then $\sum_{|J|=k} d_J \equiv \binom{s-t-1}{k-t}d_T \pmod{g_S}$.
\end{prop}
\begin{proof}
	Recall that $T \subseteq J \subset S$ for all $J$ so that $g_J \in \mathcal{L}_{m, S, T}$.
	If $|J |=k$, $d_J \equiv \sum_{i \in S \setminus J} d_{S \setminus\{i\} }\pmod{g_S}$, by Proposition \ref{sum_idems_gen}. The terms $S \setminus \{i\}$ are the $(s-1)$-element subsets of $S$ with $i \in S \setminus T$. Notice that each $d_{S \setminus\{i\}}$ is in exactly $\binom{s-t-1}{k-t}$ of the terms in the summation. By Proposition \ref{sum_idems_gen}, 
	$$
	\sum_{i \in S \setminus T} d_{S \setminus \{i\}} \equiv d_T \pmod{g_S}.
	$$
	Therefore the summation reduces to $$\sum_{|J|=k} d_J \equiv \binom{s-t-1}{k-t}d_T \pmod{g_S}.$$
\end{proof}

\begin{lemma}
	\label{one_level_gen}
	Let $S \ne \emptyset$, $\mathcal{L}_{m,S,T}\subseteq \mathcal{L}_m$. Suppose $T \subset I \subset S$ and set $|T| = t, |I| = k$. We sum over all $d_J$ so that $|J| = k-n$ and $T \subset J \subset I$ where $0< n< k-t$. In other words, we sum over all elements in the idempotent lattice $n$ levels below $d_I$ that eventually connect to $d_I$. This gives the identity
	$$\sum_{\substack{|J| = k-n\\ T \subset J \subset I}} d_J  \equiv \binom{k-t-1}{n}d_T + \binom{k-t-1}{n-1} d_I\pmod{g_S}.$$
\end{lemma}
\begin{proof} 
	By Proposition \ref{sum_idems_gen},
	\begin{align*}
	\sum_{\substack{|J| = k-n\\T\subset J \subset I}} d_J &= \sum_{\substack{|J| = k-n\\T \subset J \subset I}}  \sum_{i \in S \setminus J} d_{S \setminus \{i\}}.
	\end{align*}
	The inner sum splits up further into
	$$
	\sum_{i \in S \setminus J} d_{S \setminus \{i\}} = \sum_{i \in S \setminus I} d_{S\setminus \{i\}} +\sum_{i \in (I \setminus J)} d_{S \setminus \{i\}}
	$$
	and that $|I \setminus J| = n$. Further, there are $\binom{k-t}{n} $ subsets $J$ satisfying the conditions $T \subset J \subset I$ such that $|I \setminus J| = n$. Therefore 
	\begin{align*}
	\sum_{\substack{|J| = k-n\\ T \subset J \subset I}} d_J  &\equiv \binom{k-t}{n} \sum_{i \in S \setminus I} d_{S\setminus \{i\}}  + \sum_{\substack{|J| = k-n\\ T \subseteq J \subset I}} \sum_{i \in ( I \setminus J)} d_{S \setminus \{i\}}\\
	&\equiv \binom{k-t}{n} d_I + 	\sum_{\substack{|J| = k-n\\ T \subset J \subset I}} \sum_{i \in (I \setminus J)} d_{S \setminus \{i\}}
	\end{align*}
	Consider this last summation. The terms are the $d_{S \setminus \{i\}}$ such that $i \in I \setminus T$. Each element $d_s$ for $s \in S \setminus \{i\}$ with $i \in I$ occurs $\frac{\binom{k-t}{n} (k-t- n)}{k-t}$ times. This is true as there are $\binom{k-t}{n}$ subsets $J$ and each such subset contains $k-t-n$ terms from $I \setminus T$. All terms occur equally often and we have overcounted by a factor of $k-t$. Notice that
	$$
	\sum_{i \in I \setminus T} d_{S \setminus\{i\}} \equiv d_{S \setminus \left(I\setminus T\right)}
	$$
	by Lemma \ref{sub_set_sum_gen}. Thus
	\begin{align*}
	\sum_{\substack{|J| = k-n\\ T \subset J \subset I}} \sum_{i \in (I \setminus J)} d_{S\setminus \{i\}}&= \frac{\binom{k-t}{n} (k-t-n)}{k-t} d_{S\setminus \left(I\setminus T\right)} .
	\end{align*}
	
	Putting these summations together, 
	\begin{align*}
	\sum_{\substack{|J| = k-n\\ T \subset J \subset I}} d_J &\equiv\binom{k-t}{ n} d_I +\frac{\binom{k-t}{n} (k-t-n)}{k-t} d_{S \setminus \left(I\setminus T\right)}   \\
	&\equiv \left(\binom{k-t}{ n} -\frac{\binom{k-t}{n} (k-t-n)}{k-t} \right) d_I +\frac{\binom{k-t}{n} (k-t-n)}{k-t} d_T  \\
	&\equiv \binom{k-t-1}{n-1}d_I + \binom{k-t-1}{n}d_T\pmod{g_S}.
	\end{align*}
\end{proof}

The following corollary is a restatement of Lemma \ref{one_level_gen} when $n=1$.

\begin{corollary}
	Let $\mathcal{L}_{m,S,T}$ be a sublattice as above and let $|I| = k$. We sum over all $d_J$ so that $|J| = k-1$ and $T \subset J\subset I$. In other words, we sum over all elements in the idempotent lattice in the level directly below $d_I$ that are connected to $d_I$. We find $$\sum_{|J| = k-1, T \subset J \subset I} d_J  \equiv (k-t-1)d_T + d_I \pmod{g_S}.$$
\end{corollary}

The following theorem gives a formula for the summation of all idempotents below $d_I$ that is in terms of $I, S$, and $T$ only.
\begin{theorem}
	Let $\mathcal{L}_{m,S,T} \subseteq \mathcal{L}_m$. Let $|I| = k$. Then $\sum_{T \subseteq J \subseteq I} d_J \equiv 2^{k-t-1} (d_T + d_I) \pmod{g_S}$. 
\end{theorem}
\begin{proof}
	By Lemma \ref{one_level_gen}, $$\sum_{\substack{|J| = k-n\\ T \subset J \subset I}} d_J \equiv \binom{k-t-1}{n}d_T + \binom{k-t-1}{n-1} d_I\pmod{g_S}.$$ When $J=T$, we see that $d_J = d_T$ and when $J = I$, $d_J = d_I$. Therefore
	\begin{align*}
	\sum_{T \subseteq J \subseteq I} d_J &\equiv d_T +d_I + \sum_{n=1}^{k-t-1} \sum_{\substack{|J| = k-n\\ T \subset J \subset I}} d_J \pmod{g_S}\\
	&\equiv d_T+ d_I +d_T\sum_{n=1}^{k-t-1} \binom{k-t-1}{n} +d_I \sum_{n=1}^{k-t-1}\binom{k-t-1}{n-1}  \pmod{g_S}.
	\end{align*}
	The result follows from identities on binomial coefficients.
	
	\begin{align*}
	\sum_{T \subseteq J \subseteq I} d_J& 
	\equiv 2^{k-t-1}d_T +2^{k-t-1} d_I  \pmod{g_S}.
	\end{align*}
\end{proof}

\begin{corollary}
	Let $\mathcal{L}_{m,S,T}\subseteq \mathcal{L}_m$. Let $|I| = k$. We find $\sum_{T \subseteq J \subseteq I} d_J \equiv 2^{k-t-1}d_T\pmod{g_I}$. 
\end{corollary}
\begin{corollary}
	If $\mathcal{L}_{m,S,T} \subseteq \mathcal{L}_m$, $\sum_{T \subseteq J \subseteq S} d_J \equiv 2^{k-t-1}d_T \pmod{g_S}.$
\end{corollary}

\section{Applications of the idempotent identities}
\label{app_sect}

\subsection{Components of $\Z / m\Z$}
\label{comps_sect}
Consider the sequence $a, a^2, a^3, \ldots $ in $\Z/ m\Z$. The sequence must be finite, so eventually $a^k = a^{\ell}$. We call $a , a^2, \ldots, a^{k-1}$ the \emph{tail} and $a^k , \ldots, a^{\ell}$ the \emph{cycle} corresponding to $a$. 

\begin{definition}
	The \emph{sequential power graph} is the graph with vertex set $V = \Z/ m\Z$ and directed edges $(a,b)$ if and only if $a \equiv c^i \pmod{m}$ and $b \equiv c^{i+1} \pmod{m}$ for some $c \in \Z/ m\Z$ and some $i \in \N$. 
\end{definition}

This graph represents the connection between these sequences $a, a^2, a^3, \ldots$ for each $a \in \Z/ m\Z$. It is closely related to the power graph, which was defined in \cite{kelarev}. The power graph has been well-studied, see for example \cite{cameron} and \cite{chakrabarty}. Much of the structure of this graph follows directly from the structure of the idempotents modulo $m$. We summarize the connection to idempotents in the below propositions from \cite{ish_comps}. This connection will be used in Section \ref{alg_sect}.

\begin{prop}\cite[Corollary 5.17, Theorem 5.18]{ish_comps}
	The connected components of the sequential power graph are composed of the roots of the idempotents in $\Z / m\Z.$
	There is a one-to-one correspondence between connected components and idempotents of $\Z / m\Z$. Therefore we can refer to a component $C_I$ as the component with idempotent $d_I$.
\end{prop}

\begin{prop}\cite[Theorem 5.21]{ish_comps}
	The cycles of connected components are themselves a connected subgraph, and are groups under multiplication. The associated idempotent $d$ is the group's unity. Further, the elements in the cycles of component $C_I$ are exactly the elements in $d_IU$ where $U = \left(\Z/ m\Z\right)^\times$.
\end{prop}

Let $a, a^2, \ldots, a^k = d$ be the sequence of distinct powers of $a \in \Z/ m\Z$. Observe that if $a$ is a cycle element, then $a^{k+1} = a$. Whenever $a$ is in a cycle, the above propositions imply that $a \in d_IU$, that is, $a = d_I u$ for some $u \in U$. The tail elements cannot be determined using $d_I$, but can be expressed in terms of a similar element.

\begin{definition} \cite{ish_comps}
	Let $I = \{i_1, \ldots, i_s\} \subseteq R$. The \emph{multiplier} corresponding to set $I$ is $\pi_I = p_{i_1} \cdots p_{i_s}$.
\end{definition}

\begin{prop} \cite[Proposition 5.2]{ish_comps}
	Let $\pi_I $ be the multiplier corresponding to set $I \subseteq R$. The elements of $C_I$ are the integers $\pi_I x$, where gcd$(x, \frac{m}{g_I})=1$. 
\end{prop}


\subsection{An efficient algorithm for modular exponentiation}
\label{alg_sect}
First, suppose that $u \in U = \left(\Z/ m\Z\right)^\times$ and we wish to compute $u^e \pmod{m}$. We give an algorithm for computing $u^e$ using the structure of the idempotents from Section \ref{idem_section_1} which is a reformulation of the Chinese Remainder Theorem.  In particular, the following theorem shows a connection between primitive orthogonal idempotents modulo $m$ and the Chinese Remainder Theorem. This connection is known, but not widely; see \cite{porubsky} for a history on the relation between idempotents and the Chinese Remainder Theorem. Further, this theorem is true in general for any $b \in \Z / m\Z$, however we state it just for the units in $\Z/ m\Z$ to make a distinction between this and the next theorem.
\begin{theorem} \cite{porubsky} \label{crt_alg}
	Let $u\in U$, the set of units, and let $d_i = d_{R \setminus \{i\}}$ be the top-level idempotents. Let $e \in \N$. Then 
	$$
	u^e \equiv \sum_{i=1}^{r} d_i u^{e \pmod{\phi(p_i^{e_i})}} \pmod{m}.	
	$$
\end{theorem}


We now generalize the algorithm to compute $b^e \pmod{m}$ for any $b \in C_I$ where $|I|< r-1$. We note that this algorithm does not apply when computing $b^e \pmod{m}$ for $C_I$ where $|I| = r-1$. That is, when the idempotent $d_I$ is a top-level idempotent. We introduce two theorems that are reductions of the previous theorem. We provide proofs of these theorems that rely on an identity from Section \ref{gen_id_add} and results from Section \ref{comps_sect}. To the best of our knowledge, these reductions are not used in practice.

\begin{theorem} \label{general_alg_cyc}
	Let $b \in C_I$ so that $b = b^k$ for some $k > 1$. In other words, $d_I b = b$. Let $T = I$ and $S = R$. Let $d_i = d_{S \setminus \{i\}}$. Then
	$$
	b^e \equiv \sum_{i \in S \setminus T} d_i b^{e \pmod{\phi(p_i^{e_i})}} \pmod{m}.
	$$
\end{theorem}
\begin{proof}
	Since $S = R$, then $g_S = m$. By Proposition \ref{sum_idems_gen},
	$$
	d_T \equiv \sum_{i \in S \setminus T} d_i \pmod{m}.
	$$
	Therefore,
	$$
	b^e \equiv d_T b^e \equiv \sum_{i \in S \setminus T} d_i b^e \pmod{m}.
	$$
	Since each $d_i$ contains all primes in the factorization of $\frac{m}{p_i^{e_i}}$, then $d_i \equiv 0 \pmod{p_j^{e_j}}$ for all $j \ne i$ and $d_i \equiv 1 \pmod{p_i^{e_i}}$ by Theorem \ref{idem_struc}. Thus $d_ib^e \equiv d_i b^{e \pmod{\phi(p_i^{e_i})}} \pmod{m}$ for each $i \in S \setminus T$. 
\end{proof}

The previous theorem only holds when $b$ is in the cycle of $C_I$. The benefit to this previous theorem is it has no restrictions on the exponent $e$. In the following, we show that the same result holds when $b$ is not in the cycle, but only when $e$ is large enough.

\begin{theorem} \label{general_alg_non_cyc}
	Set $S = R$. Let $T\subseteq S$. Let $b\in C_{T}$, that is, let $b = \pi_T x$ for some $x$ relatively prime to $\frac{m}{g_T}$.  Let $d_{j} = d_{S\setminus \{j\}} $ be the top-level idempotents, for $1 \le j \le r$. Let $e \geq \max(e_1, \ldots, e_r)$. Then 
	$$
	b^e \equiv \sum_{i \in S \setminus T} d_i b^{e \pmod{\phi(p_i^{e_i})}} \pmod{m}.	
	$$
\end{theorem}
\begin{proof}
	Since $e \geq \text{max}(e_1, \ldots, e_r)$ and $b^e = \pi_T^e x^e$, then $d_T | \pi^e$.  Thus $d_T | b^e$ and so $d_T b^e \equiv b^e \pmod{m}$. The result follows by applying Theorem \ref{general_alg_cyc}.
\end{proof}
\begin{remark}
	Note that we may have $e \pmod{\phi(p_i^{e_i})} = 0$. In this case, $b^0$ is considered as an element in $d_i U$. Since the identity of $d_iU$ is $d_i$, then here $b^0 = d_i$. 
\end{remark}

\begin{remark}
	We have throughout used the Euler totient function. We can, however, use the Carmichael totient function for a tighter calculation. 
\end{remark}

\section{Acknowledgements}
A portion of this work was performed at the Ultrascale Systems Research Center (USRC) at Los Alamos National Laboratory, supported by the U.S. Department of Energy contract DE-FC02-06ER25750. The first author was supported in part by an appointment with the National Science Foundation (NSF) Mathematical Sciences Graduate Internship (MSGI) Program sponsored by the NSF Division of Mathematical Sciences. 

\bibliographystyle{hplain}
\bibliography{./biblio}

\end{document}